\documentclass[11pt]{amsart}
\usepackage{amsmath,amssymb,mathrsfs}
\usepackage[colorlinks, linkcolor=red, citecolor=green,filecolor=magenta,urlcolor=cyan]{hyperref}

\allowdisplaybreaks

\newcommand{\abs}[1]{\left\lvert#1\right\rvert}
\newcommand{\kh}[1]{\left(#1\right)}%
\newcommand{\dkh}[1]{\left\{#1\right\}}%
\newcommand{\jkh}[1]{\left\langle #1\right\rangle}

\newtheorem{theorem}{Theorem}
\newtheorem{corollary}[theorem]{Corollary}
\newtheorem{lemma}[theorem]{Lemma}

\newtheorem{prop}[theorem]{Proposition}

\theoremstyle{remark}
\newtheorem{rem}{Remark}[section]

\begin{document}

\title[Cartan  hypersurface in $\mathbb S^5$]{A curvature characterization of the Cartan minimal hypersurface in  $\mathbb S^5$}

\author{Qing Cui}
\address{School of Mathematics \\ Southwest Jiaotong University \\ 611756 Chengdu \\ Sichuan \\ China} 
\email{cuiqing@swjtu.edu.cn}

 \begin{abstract} 
Lawson showed that a non-totally geodesic Einstein minimal hypersurface in $\mathbb S^5$ is congruent to the Clifford hypersurface $
    \mathbb S^2(1/\sqrt2)\times \mathbb S^2(1/\sqrt2).
$
It is also known, by work of Cartan and \^{O}tsuki, that a non-totally geodesic locally conformally flat minimal hypersurface in $\mathbb S^5$ is of \^{O}tsuki type, including the Clifford hypersurface
$
    \mathbb S^1(1/2)\times \mathbb S^3(\sqrt3/2).
$
In this paper we study closed minimal hypersurfaces
$M$ in $\mathbb S^5$ satisfying
$
    |W|^2=2|\mathring{\operatorname{Ric}}|^2,
$
where $W$ is the Weyl tensor and $\mathring{\operatorname{Ric}}$ is the trace-free Ricci tensor.  We call this the Euler-balanced condition.  We prove that such a hypersurface is either totally geodesic or congruent to the Cartan minimal hypersurface.
\end{abstract}
\maketitle

\section{Introduction}

Minimal hypersurfaces in the unit sphere form one of the central objects in submanifold geometry.  Even in the special ambient sphere $\mathbb S^5$, the class of closed minimal hypersurfaces is already remarkably rich. Among the classical and explicitly describable examples are the totally geodesic sphere $\mathbb S^4$, the Clifford minimal hypersurfaces
\begin{align*}
   \mathbb S^k\left(\sqrt{\frac{k}{4}}\right)
    \times
    \mathbb S^{4-k}\left(\sqrt{\frac{4-k}{4}}\right),
    \quad
    1\le k\le 3,
\end{align*} 
the Cartan minimal hypersurface with four distinct principal curvatures, and the \^{O}tsuki-type minimal hypersurfaces with two distinct principal curvatures. Besides these classical and explicitly describable examples, there are many further closed minimal hypersurfaces in spheres.  Hsiang and Lawson \cite{HsiLaw71} developed an equivariant construction of minimal submanifolds of low cohomogeneity, and subsequent work of Hsiang and his
collaborators produced non-homogeneous minimal hyperspheres in spheres \cite{Hsi83a, Hsi83b, HsiSte86, HsiHsiTom88}.  From a different point of view, Song's min--max theorem \cite{Son23} shows that any closed Riemannian manifold of dimension $3\le n\le7$ contains infinitely many smoothly embedded closed minimal hypersurfaces; this applies in particular to the round sphere $\mathbb S^5$.   This abundance makes the classification of closed minimal hypersurfaces in $\mathbb S^5$ a subtle problem, and it is natural to seek rigidity under additional curvature assumptions.  The dimension four of the hypersurface is also special from the intrinsic point of view, because the Weyl tensor and the Gauss--Bonnet--Chern formula enter the geometry in an essential way.

Let $M$ be a closed minimal hypersurface in $\mathbb S^5$.  We denote by $h$ the second fundamental form and by $A$ the shape operator, so that $h(X,Y)=\langle AX,Y\rangle.$ We set $f_k=\operatorname{tr}A^k$. Two classical intrinsic curvature conditions on $M$ can be expressed purely in terms of $|A|^2$ and $f_4$.  Indeed, the Gauss equation gives
$
    R=12-|A|^2,
$
where $R$ is the scalar curvature of $M$.  Moreover (see Lemma \ref{lem:curvature-formulas}),  
\begin{align*}
 |\mathring{\operatorname{Ric}}|^2=f_4-\frac14|A|^4, \quad |W|^2=
    \frac{7}{3}|A|^4-4f_4,
\end{align*}
where $W$ is the Weyl tensor and
$\mathring{\operatorname{Ric}}$ is the trace-free Ricci tensor.  Consequently,
\begin{align*}
 \frac14|A|^4\le f_4\le \frac{7}{12}|A|^4.
\end{align*}
The lower endpoint
$
    f_4=\frac14|A|^4
$
is precisely the Einstein condition
$
   \mathring{\operatorname{Ric}}=0.
$
The upper endpoint
$
    f_4=\frac{7}{12}|A|^4
$
is precisely the locally conformally flat condition
$
    W=0.
$
These two endpoint geometries are already well understood. 
Actually, by Lawson's classification of minimal Einstein hypersurfaces in the unit sphere \cite{Law69}, a
four-dimensional minimal Einstein hypersurface in $\mathbb S^5$ is either
totally geodesic or congruent to
$
    \mathbb S^2\left(\frac1{\sqrt2}\right)
    \times
    \mathbb S^2\left(\frac1{\sqrt2}\right).
$
On the other hand, by Cartan's theorem on conformally flat hypersurfaces \cite{Car1917} and \^{O}tsuki's classification of minimal hypersurfaces with two principal curvatures \cite{Ots70}, a non-totally geodesic closed locally conformally flat minimal hypersurface in
$\mathbb S^5$ is of \^{O}tsuki type, with the constant-principal-curvature case given by
$
    \mathbb S^1\left(\frac12\right)
    \times
    \mathbb S^3\left(\frac{\sqrt3}{2}\right).
$

The purpose of the present paper is to study a natural intermediate condition between the Einstein condition and the locally conformally flat condition.  We shall say that a four-dimensional Riemannian manifold is \emph{Euler-balanced}
if
 $
    |W|^2=2|\mathring{\operatorname{Ric}}|^2.
$
For a minimal hypersurface \(M^4\subset\mathbb S^5\), this condition is
equivalent to
$
    f_4=\frac{17}{36}|A|^4.
$
 It is not merely an
algebraic number between $1/4$ and $7/12$; rather, it is the unique value for
which the Weyl part and the trace-free Ricci part balance in the
Gauss--Bonnet--Chern integrand.  In fact, the four-dimensional
Gauss--Bonnet--Chern formula reads
\begin{align*}
    32\pi^2\chi(M)=\int_M \left( |W|^2 -2|\mathring{\operatorname{Ric}}|^2+\frac16R^2 \right)dV.
\end{align*}
Therefore, under the Euler-balanced condition, one obtains
\begin{equation}\label{Euler-balanced-total-curvature}
\chi(M)=\frac{1}{192\pi^2}\int_M R^2\,dV.
\end{equation}
Thus the Euler-balanced condition is intrinsically tied to the Euler characteristic of $M$.

Our main theorem shows that, in the minimal hypersurface setting, the Euler-balanced condition is in fact rigid.
\begin{theorem}\label{thm:main}
Let $M$ be a closed connected oriented minimal hypersurface in $\mathbb S^5$. Assume that
 $
    |W|^2=2|\mathring{\operatorname{Ric}}|^2.
$
Equivalently,
$
    f_4=\frac{17}{36}|A|^4.
$
Then $M$ is either totally geodesic or congruent to the Cartan minimal hypersurface in $\mathbb S^5$.
\end{theorem}

Since the Cartan hypersurface is the minimal isoparametric hypersurface in $\mathbb S^5$ with four distinct principal curvatures, Theorem \ref{thm:main} is also related to the strong version of Chern's conjecture (for the original Chern conjecture, see \cite{ChedoCKob70}), which asserts that any closed minimal hypersurface in $\mathbb S^{n+1}$ with constant $|A|^2$ should be isoparametric.  This conjecture is proved to be true in dimension $n=3$ by Chang \cite{Cha93}, while in higher dimensions it remains widely open.  In dimension four,  recent preprints  \cite{HeXuZha26, GLLY26, DenKou26} have also investigated related rigidity problems under additional trace conditions,  such as the constancy of $f_3$ or $f_4$.

Our theorem is different in nature.  We do not assume that $|A|^2$ is constant. Instead, we impose the intrinsic pointwise   condition
$
    |W|^2=2|\mathring{\operatorname{Ric}}|^2.
$
The conclusion then forces
$|A|^2\equiv0$ or $|A|^2\equiv 12$. Thus the theorem may be viewed as a rigidity result outside the usual constant scalar curvature framework of Chern's conjecture: a pointwise balance between
the Weyl curvature and the trace-free Ricci curvature is strong enough to force the same kind of isoparametric rigidity predicted by Chern's conjecture.

We also want to point out that, the terminology ``Euler-balanced'' is meant in dimension four.  In dimension
two the condition is trivial, while in dimension three the Weyl tensor vanishes identically and the condition reduces to the Einstein condition.  Hence the nontrivial balance between the Weyl curvature and the trace-free Ricci curvature is a genuinely four-dimensional phenomenon.

In dimension four, there is also a natural connection with the Yamabe constant and the Euler-balanced condition. For any metric $g$ with $Y(M,[g])\ge0$, the Cauchy--Schwarz inequality and \eqref{Euler-balanced-total-curvature} give
\begin{equation*}
 Y(M,[g])^2  \le 192\pi^2\chi(M).
\end{equation*}
In particular, if $\chi(M)=0$, then $R_g\equiv0$ and $Y(M,[g])=0$.

\begin{rem}
The condition
$
    |W|^2=c|\mathring{\operatorname{Ric}}|^2
$
does not seem to define a standard class of four-dimensional Riemannian metrics.  It is related, however, to the extensive literature on pointwise and integral curvature pinching involving the Weyl tensor and the trace-free Ricci
tensor; see, for instance, \cite{Hui85, Mar86, Mar98, Tra17}.
In the present paper we use the special value $c=2$, for which the non-scalar part of the four-dimensional  Gauss--Bonnet--Chern integrand vanishes. This gives the condition a particularly natural meaning in dimension four.

To the best of our knowledge, the pointwise equality
$
    |W|^2=2|\mathring{\operatorname{Ric}}|^2
$
has not been studied as an independent rigidity condition for minimal hypersurfaces in spheres.  Our result shows that, in the hypersurface setting, this intrinsic curvature balance is strong enough to force the Cartan alternative.
\end{rem}

The paper is organized as follows. In Section 2, we review some notations and facts about the submanifold geometry, especially in dimension 4, and give several elementary lemmas which will be used in Section 3. In Section 3, we give a proof of Theorem \ref{thm:main}. The proof divides into 3 main steps: First we show that under the Euler balanced condition, the existence of an umbilic  point  forces $M$ to be totally geodesic. Second, we prove that the absence of umbilic  points implies $\chi(M)=0$; Third, we show that under the condition $|A|^2\equiv 12$ and $f_4\equiv 68$, $ M$ must be the Cartan minimal hypersurface. It should be mentioned that, very recently,  there are two recent preprints \cite{GLLY26, DenKou26} dealing with the case $|A|^2, f_4$ are constants. For completeness and clarity,  we give an alternative and shorter proof of the rigidity of  the special case: $|A|^2\equiv 12$ and $f_4\equiv 68$.

\vspace{1cm}

\section{Preliminaries}\label{sec:preliminaries} 
Let $M$ be  a  closed minimal hypersurface in the unit sphere $\mathbb{S}^5$ with second fundamental form $h$ and shape operator $A$.  We use the convention that repeated indices are summed from $1$ to $4$. Let $e_1,\cdots,e_4$ be a local orthonormal frame on $M$, and let $\omega_1,\cdots,\omega_4$ be its dual coframe. Write 
\begin{align*}
h=\sum_{i,j}h_{ij}\,\omega_i\otimes\omega_j, \qquad h_{ij}=h_{ji}. 
\end{align*}
 Since \(M\) is minimal, $ \sum_i h_{ii}=0. $ The covariant derivatives of $h$ are denoted by 
\begin{align*}
h_{ijk}=\nabla_k h_{ij}, \quad h_{ijkl}=\nabla_l\nabla_k h_{ij}.
\end{align*} 
Equivalently,
\begin{align}\label{hijkdef}
\sum_k h_{ijk}\omega_k = dh_{ij} + \sum_k h_{kj}\omega_{ki} + \sum_k h_{ik}\omega_{kj}.
\end{align} 
The structure equations are 
\begin{align}
d\omega_i=&\sum_j\omega_{ij}\wedge\omega_j, \qquad \omega_{ij}+\omega_{ji}=0, \label{domegai} \\
 d\omega_{ij} =& \sum_k\omega_{ik}\wedge\omega_{kj} - \frac12\sum_{k,l}R_{ijkl}\omega_k\wedge\omega_l. \label{domegaij} 
\end{align} 
With these conventions, the Gauss equation is 
\begin{align}\label{Gausseq}
R_{ijkl} = \delta_{ik}\delta_{jl} - \delta_{il}\delta_{jk} + h_{ik}h_{jl} - h_{il}h_{jk}.
\end{align}
 The Codazzi equation is 
\begin{align}\label{Codazzieq}
h_{ijk}=h_{ikj}. 
\end{align} 

For $k\ge1$, define $f_k:=\operatorname{tr}A^k=\sum_i\lambda_i^k. $ We also write $K:=\lambda_1\lambda_2\lambda_3\lambda_4 $ for the Gauss--Kronecker curvature. 
\begin{lemma}\label{lem:newton-identities} For a minimal hypersurface $M\subset\mathbb S^5$, one has $f_1=0,\  f_2=|A|^2, $ and 
\begin{align}\label{f4K}
f_4=\frac12|A|^4-4K.
\end{align} 
    Equivalently, the characteristic polynomial of $A$ is 
\begin{align}\label{polynomialofA}
 x^4-\frac{|A|^2}{2}x^2-\frac{f_3}{3}x+K.
\end{align} 
 \end{lemma}
 \begin{proof} Let $\sigma_k$ be the $k$-th elementary symmetric polynomial of $\lambda_1,\cdots,\lambda_4$. Since $M$ is minimal,  $\sigma_1=f_1=0. $ Newton's identities give $|A|^2=f_2=\sigma_1^2-2\sigma_2=-2\sigma_2,$ and 
\begin{align*}
f_3=\sigma_1^3-3\sigma_1\sigma_2+3\sigma_3=3\sigma_3.
\end{align*}
 Moreover, 
\begin{align*}
f_4 = \sigma_1^4 - 4\sigma_1^2\sigma_2 + 4\sigma_1\sigma_3 + 2\sigma_2^2 - 4\sigma_4 = 2\sigma_2^2-4\sigma_4.
\end{align*} 
 Since $\sigma_2=-\frac12 |A|^2$ and $\sigma_4=K$, we obtain 
\begin{align*}
f_4=\frac12|A|^4-4K.
\end{align*} 
Finally, the characteristic polynomial is 
\begin{align*}
x^4-\sigma_1x^3+\sigma_2x^2-\sigma_3x+\sigma_4,
\end{align*} 
 which, using $
\sigma_1=0,\  \sigma_2=-\frac{|A|^2}{2}, \ \sigma_3=\frac{f_3}{3}, \  \sigma_4=K,$  becomes 
\begin{align*}
x^4-\frac{|A|^2}{2}x^2-\frac{f_3}{3}x+K. 
\end{align*} 
\end{proof}

\begin{lemma}\label{lem:curvature-formulas} Let $M \subset \mathbb S^5$ be minimal. Let $R, W, \operatorname{Ric}, \mathring{\operatorname{Ric}}$ be the scalar curvature, the Weyl curvature tensor, the Ricci curvature tensor and the trace-free of the Ricci curvature tensor of $M$. Then
\begin{align}
R=&12-|A|^2,\label{scalarcurv}\\
 |W|^2 =& \frac{7}{3}|A|^4-4f_4,\label{Weylcurv}\\
 |\operatorname{Ric}|^2 =& 36-6|A|^2+f_4,\label{Riccicurv}\\
 |\mathring{\operatorname{Ric}}|^2 =& f_4-\frac14 |A|^4.\label{TFRiccicurv}
\end{align}
 \end{lemma} 
\begin{proof} A direct computation by using the Gauss equation \eqref{Gausseq} and the formulas of Weyl curvature, Ricci curvature.
\end{proof}
 \begin{corollary}\label{cor:endpoint-balanced} For a minimal hypersurface \(M^4\subset\mathbb S^5\), one has 
\begin{align}\label{f_4range}
\frac14|A|^4\le f_4\le \frac{7}{12}|A|^4.
\end{align} 
 Moreover, 
\begin{align}
f_4=\frac14|A|^4 \quad\Longleftrightarrow&\quad M\text{ is Einstein},\label{Einstein}\\
 f_4=\frac{7}{12}|A|^4 \quad\Longleftrightarrow&\quad M\text{ is locally conformally flat.}
\end{align}
The Euler-balanced condition $ |W|^2=2|\mathring{\operatorname{Ric}}|^2$ is equivalent to $f_4=\frac{17}{36}|A|^4.$ \end{corollary} 
\begin{proof}
By Lemma \ref{lem:curvature-formulas},
$
|\mathring{\operatorname{Ric}}|^2=f_4-\frac14|A|^4.
$
Thus $f_4\ge \frac14|A|^4$, and equality holds if and only if
$\mathring{\operatorname{Ric}}=0$, namely $M$ is Einstein.

Similarly,
$
|W|^2=\frac73|A|^4-4f_4.
$
Hence $f_4\le \frac{7}{12}|A|^4$, and equality holds if and only if $W=0$. Since $\dim M=4$, $W=0$ is equivalent to local conformal flatness.

Finally,
$
|W|^2=2|\mathring{\operatorname{Ric}}|^2
$
is equivalent to
$
\frac73|A|^4-4f_4
=
2\left(f_4-\frac14|A|^4\right),
$
which gives
$
f_4=\frac{17}{36}|A|^4.
$
\end{proof}
\begin{lemma}\label{lem:GBC} Let \(M^4\subset\mathbb S^5\) be a closed minimal hypersurface. Then the Gauss--Bonnet--Chern formula becomes
\begin{align*} 
16\pi^2\chi(M) = \int_M \left( \frac32|A|^4-3f_4-2|A|^2+12 \right)\,dV.
\end{align*} 
 In particular, if $M$ is Euler-balanced,   then 
\begin{align}\label{GBC-Euler}
 \chi(M) = \frac{1}{192\pi^2} \int_M(|A|^2-12)^2\,dV. 
\end{align}  
\end{lemma} 
\begin{proof} For a closed four-dimensional Riemannian manifold, the Gauss--Bonnet--Chern formula can be written as (cf. \cite{Bes87}) 
\begin{align*}
16\pi^2\chi(M) = \int_M \left( \frac{R^2}{3} - | \operatorname{Ric} |^2 + \frac{|W|^2}{2} \right)\,dV.
\end{align*}
 Substituting   identities in Lemma \ref{lem:curvature-formulas} gives 
 \begin{align*}
\frac{R^2}{3} - |\operatorname{Ric}|^2 + \frac{|W|^2}{2} &= \frac{(12-|A|^2)^2}{3} - (36-6|A|^2+f_4) + \frac12\left(\frac73|A|^4-4f_4\right) \\ &= \frac32|A|^4-3f_4-2|A|^2+12. 
\end{align*}
 If $M$ is Euler-balanced, we have $ f_4=\frac{17}{36}|A|^4, $  then we obtain \eqref{GBC-Euler}.
\end{proof}
 \begin{lemma}\label{lem:euler-balanced-algebra} Let \(\lambda_1,\cdots,\lambda_4\) be the principal curvatures of a minimal hypersurface $M\subset\mathbb S^5$ at a point. If $ f_4=\frac{17}{36}|A|^4, $ then $ K=\frac{|A|^4}{144}.$
 Consequently, if $|A|^2>0$, then $K>0$, and exactly two principal curvatures are negative and two are positive. Moreover, at such a point with $|A|^2>0$, the multiplicity types $2+2$ and $3+1$ cannot occur. Hence either the four principal curvatures are distinct, or there is exactly one double principal curvature. 
\end{lemma} 
\begin{proof} The assumption $f_4=\frac{17}{36}|A|^4$  and Lemma \ref{lem:newton-identities}   give  $ K=\frac{|A|^4}{144}. $ 
If $|A|^2>0$, then $K>0$. Since $ \sum_i\lambda_i=0, $ the four nonzero numbers $\lambda_i$ cannot all have the same sign. Since their product is positive, the number of negative ones is even. Hence exactly two are negative and two are positive. It remains to exclude the multiplicity types. If the type is $2+2$, then, by minimality, the principal curvatures are $ a,a,-a,-a $ with $a\ne0$. Then 
\begin{align*}
 \frac{f_4}{|A|^4} = \frac{4a^4}{(4a^2)^2} = \frac14,
\end{align*}
which contradicts $\frac{f_4}{|A|^4}=\frac{17}{36}. $
 If the type is $3+1$, then the principal curvatures are $a,a,a,-3a$ with $a\ne0$. Then 
\begin{align*}
\frac{f_4}{|A|^4} = \frac{3a^4+81a^4}{(3a^2+9a^2)^2} = \frac{84}{144} = \frac{7}{12},
\end{align*} 
again a contradiction. Therefore the only possible multiplicities are the distinct-root case and the case of exactly one double root. \end{proof} 

\begin{lemma}\label{lem:S12_f4_68_algebra} Assume that real numbers $\lambda_1,\cdots,\lambda_4$ satisfy
\begin{align*}
 \sum_i\lambda_i=0, \quad \sum_i\lambda_i^2=12, \quad \prod_i\lambda_i=1.
\end{align*} 
Set $f_3:=\sum_i\lambda_i^3. $ Then  $|f_3|\le 8\sqrt{2\sqrt3}.$ Moreover, 
\begin{itemize}
\item  if $|f_3|< 8\sqrt{2\sqrt3}$, then the four roots are distinct.
\item if $f_3=8\sqrt{2\sqrt3}$, then, after ordering,   $\lambda_1=\lambda_2<0<\lambda_3<\lambda_4.$
\item if $f_3=-8\sqrt{2\sqrt3}$, then, after ordering, $ \lambda_1<\lambda_2<0<\lambda_3=\lambda_4. $ 
\end{itemize}

  \end{lemma}
 \begin{proof}Assume $\sigma_k$ is the $k$-th symmetric function of $\lambda_1, \cdots, \lambda_4$. Newton's identities give 
\begin{align*}
\sigma_2=-6, \quad \sigma_3=\frac{f_3}{3}, \quad \sigma_4=1. 
\end{align*}
 Therefore the characteristic polynomial is $ x^4-6x^2-\frac{f_3}{3}x+1.$  It remains to determine the possible range of $f_3$. The set of quadruples 
\begin{align*}
\left\{ (\lambda_1,\cdots,\lambda_4)\in\mathbb R^4: \sum_i\lambda_i=0,\, \sum_i\lambda_i^2=12,\, \prod_i\lambda_i=1 \right\}
\end{align*} 
is compact. Hence $f_3$ attains its maximum and minimum. 
We claim that at an extremal point the four numbers
\(\lambda_1,\ldots,\lambda_4\) cannot be pairwise distinct.

Indeed, consider the functions
\begin{align*}
 F(\lambda)=\sum_{i=1}^4\lambda_i^3,\qquad
    G_1(\lambda)=\sum_{i=1}^4\lambda_i,\qquad
    G_2(\lambda)=\sum_{i=1}^4\lambda_i^2,\qquad
    G_3(\lambda)=\prod_{i=1}^4\lambda_i .
\end{align*}
The constraints are
$
    G_1=0,\  G_2=12,\ G_3=1.
$
Suppose, for contradiction, that at an extremal point the four
$\lambda_i$'s are pairwise distinct.  We first note that the gradients
$
    \nabla G_1,\ \nabla G_2,\  \nabla G_3
$
are linearly independent at this point.  In fact, if
$
    a\nabla G_1+b\nabla G_2+c\nabla G_3=0,
$
then for each $i$,
\begin{align*}
  a+2b\lambda_i+c\prod_{j\neq i}\lambda_j=0.
\end{align*} 
Since $\prod_i\lambda_i=1$, we have
$
    \prod_{j\neq i}\lambda_j=\frac1{\lambda_i}.
$
Thus
$
    a+2b\lambda_i+\frac{c}{\lambda_i}=0.
$
Multiplying by \(\lambda_i\), we get
\begin{align*}
    2b\lambda_i^2+a\lambda_i+c=0,
    \qquad i=1,\ldots,4.
\end{align*} 
Hence the four distinct numbers $\lambda_1,\ldots,\lambda_4$ are roots of the quadratic polynomial
$
    2bx^2+ax+c.
$
This is impossible unless this polynomial is identically zero.  Therefore
$
    a=b=c=0.
$
So the constraint gradients are linearly independent.

By the Lagrange multiplier method, there exist constants
\(\mu,\nu,\rho\in\mathbb R\) such that 
\begin{align*}
    \nabla F
    =
    \mu\nabla G_1+\nu\nabla G_2+\rho\nabla G_3.
\end{align*} 
Here $\nabla$ stands for the gradient of a function with respect to the variables $\lambda_1,\cdots, \lambda_4$.
For each $i$, this gives
$
    3\lambda_i^2
    =
    \mu+2\nu\lambda_i+\rho\prod_{j\neq i}\lambda_j.
$
Again using \(\prod_i\lambda_i=1\), we obtain
$
    3\lambda_i^2
    =
    \mu+2\nu\lambda_i+\frac{\rho}{\lambda_i}.
$
Multiplying by $\lambda_i$, we get
$
  3\lambda_i^3-2\nu\lambda_i^2-\mu\lambda_i-\rho=0.
$
Thus  the cubic polynomial equation
$
    3x^3-2\nu x^2-\mu x-\rho=0
$
has four distinct roots: $\lambda_1, \cdots, \lambda_4$.  This is impossible.

 Thus an extremal quadruple has a multiple root. Let $\alpha$ be a multiple root of $p(x)=x^4-6x^2-\frac{f_3}{3}x+1.$ Then
\begin{align*}
p (\alpha)=0, \qquad p'(\alpha)=0.
\end{align*}
Since $p'(x)=4x^3-12x-\frac{f_3}{3},$ we get 
\begin{align}
f_3=12\alpha^3-36\alpha=12\alpha(\alpha^2-3). \label{f_3eq}
\end{align}  
Substituting this into $p(\alpha)=0$ gives $ -3\alpha^4+6\alpha^2+1=0. $ 
Thus $ \alpha^2=1+\frac{2\sqrt3}{3}.$ By \eqref{f_3eq}, we have $|f_3|=8\sqrt{2\sqrt3}$
 at any extremal point.   If $|f_3|<8\sqrt{2\sqrt3}$, then $p(x)$ has no multiple root, and hence the four \(\lambda_i\)'s are distinct. If $f_3=8\sqrt{2\sqrt3}$, then $\alpha<0$ in \eqref{f_3eq}, so the double root is negative. Since $\prod_i\lambda_i=1>0$ and $ \sum_i\lambda_i=0, $ the signs are two negative and two positive; hence, after ordering, $ \lambda_1=\lambda_2<0<\lambda_3<\lambda_4. $ Similarly, if $f_3=-8\sqrt{2\sqrt3}$, then the double root is positive, and $ \lambda_1<\lambda_2<0<\lambda_3=\lambda_4. $
 \end{proof} 
We shall also use the following standard differential identities. Simons' identity \cite{Sim68} gives
\begin{align}
\frac12\Delta |A|^2 = |\nabla A|^2+(4-|A|^2)|A|^2.\label{Simonsid}
\end{align} 


We record the following Lewy--Gleason--Wolff lemma for Hessian determinants of harmonic functions:
\begin{lemma}[\cite{GleWol91}]\label{lem:LGW}
Let $u$ be a harmonic function on a domain in $\mathbb R^n$. If $\det\nabla^2u$ does not change sign in the domain, then either $\det\nabla^2u$ is identically zero, or it is nowhere zero.
\end{lemma}
 
The following estimate is a standard form of the layer argument used by de Almeida--Brito \cite{deABri90}.
\begin{lemma}\label{lem:layer_estimate}
Let $u\in C^\infty(M)$ be a smooth function on a closed Riemannian manifold $M$.  Let
\begin{equation*}
   m=\min_M u,\qquad M_0=\max_M u.
\end{equation*}
Then
\begin{equation*}
    \lim_{\varepsilon\to0^+}
    \frac1\varepsilon
    \int_{\{m\le u\le m+\varepsilon\}}
    |\nabla u|^2\,dV
    =
    0,
\end{equation*}
and
\begin{equation*}
    \lim_{\varepsilon\to0^+}
    \frac1\varepsilon
    \int_{\{M_0-\varepsilon\le u\le M_0\}}
    |\nabla u|^2\,dV
    =
    0.
\end{equation*}
\end{lemma}

\begin{proof} For completeness, we recall the
short proof.

For almost every $t\in(m,m+\varepsilon)$, $t$ is a regular value of $u$.
By the coarea formula,
\begin{equation*}
\int_{\{m\le u\le m+\varepsilon\}}|\nabla u|^2\,dV
=
\int_m^{m+\varepsilon}
\left(
\int_{\{u=t\}}|\nabla u|\,d\sigma_t
\right)dt.
\end{equation*}
For such $t$, applying the divergence theorem to the set
$\{m\le u\le t\}$ gives
\begin{equation*}
\int_{\{u=t\}}|\nabla u|\,d\sigma_t
\le
\int_{\{m\le u\le t\}}|\Delta u|\,dV.
\end{equation*}
Hence
\begin{equation*}
\frac1\varepsilon
\int_{\{m\le u\le m+\varepsilon\}}|\nabla u|^2\,dV
\le
\int_{\{m\le u\le m+\varepsilon\}}|\Delta u|\,dV.
\end{equation*}
Since $|\Delta u|\in L^1(M)$ and the sets
$\{m\le u\le m+\varepsilon\}$ shrink to the minimum set of $u$, the right hand side tends to $0$ as $\varepsilon\to0^+$.  The estimate near the maximum follows by applying the same argument to $-u$.
\end{proof}


\vspace{1cm}

\section{Proof of Theorem \ref{thm:main}}

We divide the proof into several subsections.

 \subsection{Umbilic points}
\begin{lemma}\label{lem:umbilic_implies_totally_geodesic}
Let $M$ be a minimal hypersurface in $\mathbb{S}^5$ with second fundamental $h$, shape operator $A$ and  $f_4={\rm tr} A^4$.  Assume that
$
    f_4=\frac{17}{36} \abs{A}^4 
$ holds at each point in $M$.  If
$A(p)=0$ at some point $p\in M$, then
$M$ is totally geodesic.
\end{lemma}

\begin{proof}
Suppose, for contradiction, that \(A(p)=0\) but \(A\not\equiv0\).
Choose geodesic normal coordinates
$
    x=(x_1,\ldots,x_4)
$
centered at \(p\).  Thus \(p\) corresponds to \(x=0\), and a point near \(p\) is
written as
$
    q=\exp_p(x).
$
Let
$
    \partial_i=\frac{\partial}{\partial x_i}.
$
We write the coordinate components of the second fundamental form as
\begin{align*}
   h_{ij}(x):=
 \jkh{ A_q\partial_i,\partial_j}_q, \quad 1\le i,j \le 4.
\end{align*}
These are smooth functions of $x$ near $0$.  Moreover, since $A(p)=0$, we have $h_{ij}(0)=0$ for all $i,j$.

We first observe that $A$ cannot vanish to infinite order at $p$.  Indeed, by Simons identity \eqref{Simonsid}, 
the second fundamental form $h$ satisfies the Simons tensor equation
\begin{equation*}
 g^{ab}\nabla_a\nabla_bh=\Delta^{\nabla}h=(4-|A|^2)h,
\end{equation*}
where $ \Delta^{\nabla}$ is the rough Laplacian on symmetric $2$-tensors. Consequently, in local coordinates the components $h_{ij}$ satisfy a  system of the form
\begin{align}\label{ellipticsys}
    g^{ab}\partial_a\partial_b h_{ij}
+P^{c,\alpha\beta}_{ij}(x)\partial_c h_{\alpha\beta}
+Q^{\alpha\beta}_{ij}(x)h_{\alpha\beta}=0,
\end{align}
where the coefficients $P^{c,\alpha\beta}_{ij}$ and
$Q^{\alpha\beta}_{ij}$ are smooth functions in the coordinate neighborhood. Indeed, they are obtained by expanding the covariant derivatives
\begin{align*}
\nabla_a h_{ij}= \partial_a h_{ij}-\Gamma_{ai}^r h_{rj}- \Gamma_{aj}^r h_{ir}.
\end{align*}
Since $g^{ab}$ is positive definite, \eqref{ellipticsys} is a linear elliptic system.  Therefore, by Aronszajn's unique continuation theorem \cite{Aro57}, if $A$ vanished to infinite order at $p$, then $A\equiv0$ on $M$, contrary to our assumption $A\not\equiv0$.

Hence $A$ has a finite order of vanishing at $p$.  Therefore there exists a smallest integer $m\ge 1$ such that
$
    \partial^\alpha h_{ij}(0)\neq0
$
for some pair $i,j$ and some multi-index $\alpha$ with $|\alpha|=m$.  Thus
\begin{align*}
 \partial^\beta h_{ij}(0)=0
    \qquad
    \text{for all }i,j\text{ and all }|\beta|<m.
\end{align*}
By Taylor's formula, for every $i,j$,
\begin{align*}
  h_{ij}(x)=B_{ij}(x)+O(|x|^{m+1}),
\end{align*}
where
\begin{align*}
 B_{ij}(x):=
    \sum_{|\alpha|=m}
    \frac{1}{\alpha!}
    \partial^\alpha h_{ij}(0)x^\alpha.
\end{align*}
Each $B_{ij}$ is either zero or a homogeneous polynomial of degree $m$, and by the choice of $m$, the tensor
$
    B=(B_{ij})
$
is not identically zero.  By the symmetry of $A$, we have  $h_{ij}(x)=h_{ji}(x)$. Therefore,
\begin{align*}
0=h_{ij}(x)-h_{ji}(x)=B_{ij}(x)-B_{ji}(x)+O(|x|^{m+1}).
\end{align*}
Since $B_{ij}$ is zero or a homogeneous polynomial of degree $m$, we then have $B_{ij}=B_{ji}$.

We next derive the equations satisfied by $B$.  In geodesic normal coordinates,
\begin{equation*}
 g_{ij}(x)=\delta_{ij}+O(|x|^2),
    \qquad
    g^{ij}(x)=\delta^{ij}+O(|x|^2),
\end{equation*}
and
$
    \Gamma_{ij}^k(x)=O(|x|).
$
Since $M$ is minimal,
$
    0=\operatorname{tr}_g A=g^{ij}h_{ij}.
$
Substituting
\begin{equation*}
    g^{ij}=\delta^{ij}+O(|x|^2),
    \qquad
    h_{ij}=B_{ij}+O(|x|^{m+1}),
\end{equation*}
we get
\begin{equation*}
    0 = \delta^{ij}B_{ij} + O(|x|^{m+1}).
\end{equation*}
The term $\delta^{ij}B_{ij}$ is homogeneous of degree $m$.  Therefore it must
vanish identically.  Hence
\begin{equation}
\sum_i B_{ii}=0.    \label{Bii}
\end{equation} 

Now consider the Codazzi equation
$
    \nabla_kh_{ij}=\nabla_jh_{ik}.
$
In coordinates,
\begin{equation*}
    \nabla_kh_{ij}=\partial_kh_{ij}-\Gamma_{ki}^{\ell}h_{\ell j} - \Gamma_{kj}^{\ell}h_{i\ell}.
\end{equation*}
Since
\begin{equation*}
    \partial_kh_{ij}=\partial_kB_{ij}+O(|x|^m), \quad \Gamma_{ki}^{\ell}h_{\ell j}=O(|x|)\,O(|x|^m)= O(|x|^{m+1}),
\end{equation*}
we obtain
\begin{equation*}
\nabla_kh_{ij} =\partial_kB_{ij}+O(|x|^m).
\end{equation*}
Similarly,
\begin{equation*}
    \nabla_jh_{ik}=\partial_jB_{ik} +O(|x|^m).
\end{equation*}
The Codazzi equation therefore gives
\begin{equation*}
0=\nabla_k h_{ij}-\nabla_j h_{ik}=  \partial_kB_{ij}-\partial_jB_{ik}
+ O(|x|^m).
\end{equation*}
But $\partial_kB_{ij}-\partial_jB_{ik}$ is a homogeneous polynomial of degree $m-1$.  Hence it
must vanish identically, and so
\begin{equation}
 \partial_kB_{ij}=\partial_jB_{ik}.   \label{CodazziB}
\end{equation}
Thus \(B\) is a symmetric trace-free Codazzi \(2\)-tensor on the Euclidean vector
space \(T_pM\simeq\mathbb R^4\).

We now show that $B$ is the Hessian of a harmonic homogeneous polynomial. Actually, define
\begin{equation*}
P(x)=\frac{1}{(m+1)(m+2)}\sum_{i,j} x_i x_j B_{ij}(x). 
\end{equation*}
Since $B_{ij}$ is homogeneous of degree $m$, direct computations by using \eqref{CodazziB} and Euler's identity for homogeneous polynomials give
\begin{equation*}
\partial_i\partial_j P=B_{ij}.
\end{equation*} 
Moreover, by \eqref{Bii},
\begin{equation*}
 \Delta P=\sum_i\partial_i\partial_iP=\sum_iB_{ii}=0.
\end{equation*}
Thus $P$ is a nonzero harmonic homogeneous polynomial of degree $m+2\ge3$ on $\mathbb R^4$, and
$
    B=\nabla^2P.
$

We now use the pointwise algebraic condition.  Observe that  the condition
$
    f_4=\frac{17}{36}|A|^4
$
is equivalent to
   $ \operatorname{tr}A^4= \frac{17}{36}(\operatorname{tr}A^2)^2.
$
In the above normal coordinates, the matrix of the shape operator is
$
    A^i_{\ j}=g^{ik}h_{kj}.
$
We then have
\begin{equation*}
 A^i_{\ j}=(\delta^{ik}+O(|x|^2))(B_{kj}+O(|x|^{m+1})=B_{ij}+O(|x|^{m+1}).
\end{equation*}
Taking the homogeneous terms of degree $4m$ in
$\operatorname{tr}A^4= \frac{17}{36}(\operatorname{tr}A^2)^2
$
therefore gives
\begin{equation}
 \operatorname{tr}B^4=\frac{17}{36}(\operatorname{tr}B^2)^2.       \label{trB4}
\end{equation}
For a trace-free \(4\times4\) symmetric matrix \(B\), Newton's identity gives
\begin{equation*}
    \operatorname{tr}B^4
    -
    \frac12(\operatorname{tr}B^2)^2
    +
    4\det B
    =
    0.
\end{equation*}
Combining this identity with \eqref{trB4}, we obtain
\begin{equation}
   \det B= \frac{1}{144}(\operatorname{tr}B^2)^2.  \label{detB}
\end{equation}
Since $B=\nabla^2P$, equation \eqref{detB} becomes
\begin{equation}
\det\nabla^2P=\frac{1}{144}|\nabla^2P|^4\ge0.     \label{dethessianP}   
\end{equation}
Furthermore, $\det\nabla^2P$ is not identically zero.  Indeed, if
$
    \det\nabla^2P\equiv0,
$
then \eqref{dethessianP} gives
$
    |\nabla^2P|^4\equiv0,
$
and hence
$
    B=\nabla^2P\equiv0,
$
contradicting the construction of $B$.

On the other hand, since $P$ is homogeneous of degree $m+2\ge3$, its Hessian
$\nabla^2P$ is homogeneous of positive degree $m$.  Hence
$
    \det\nabla^2P(0)=0.
$
By Lemma \ref{lem:LGW}, we have $\det\nabla^2 P\equiv 0$.

Therefore,   we get a contradiction. This contradiction shows that our assumption $A\not\equiv0$ is impossible.
Thus, we have   $ A\equiv0$ on $M$ and $M$ is totally geodesic.
\end{proof}


\vspace{0.5cm}
\subsection{The non-umbilic case}

\begin{lemma}\label{lem:no_umbilic_chi_zero}
Assume that $M$ is a closed  oriented minimal hypersurface in $\mathbb S^5$  satisfying
$
    f_4=\frac{17}{36}|A|^4.
$
If $M$ has no umbilic point, then
$
    \chi(M)=0.
$
\end{lemma}
\begin{proof}
Since $M$ has no umibilical point, we have $|A|^2>0$ 
everywhere on $M$. By Lemma \ref{lem:newton-identities}, we have 
$
    K=\frac{|A|^4}{144}>0.
$
Thus every principal curvature is nonzero. Since $M$ is minimal,
$
    \sum_{i=1}^4\lambda_i=0,
$
so at every point there are exactly two negative and two positive principal curvatures.

Let
$
    E_- \subset TM
$
be the rank $2$ distribution spanned by the negative principal directions, and let
$
    E_+ \subset TM
$
be the rank $2$ distribution spanned by the positive principal directions. The spectral gap at $0$ implies that $E_+$ and  $E_-$ are both smooth  subbundles. Moreover,
\begin{equation}
TM=E_-\oplus E_+. \label{TMdecom}
\end{equation}
Define two subsets 
\begin{equation*}
 U_-:=\{p\in M:\lambda_1(p)<\lambda_2(p)<0\},\quad U_+:=\{p\in M:0<\lambda_3(p)<\lambda_4(p)\}.
\end{equation*}
Lemma \ref{lem:euler-balanced-algebra} implies
\begin{equation}
U_-\cup U_+=M. \label{Upm}
\end{equation}

We first assume that $E_-$ and $E_+$ are all orientable.   On $U_-$, the bundle $E_-$ splits as a direct sum of two real line bundles. Hence its Euler class vanishes in $H^2(U_-;\mathbb R)$. Indeed, after passing to a double cover on which one line summand is orientable, the pull-back of $E_-$ has a nowhere-zero section; hence its Euler class is torsion and therefore vanishes over $\mathbb R$.
Hence the Euler class of $E_-$ restricts trivially to $U_-$, i.e.,
$
    e(E_-)|_{U_-}=0.
$
Similarly, $
    e(E_+)|_{U_+}=0.
$
By relative cohomology, $e(E_-)$ comes from $H^2(M,U_-;\mathbb R)$, while $e(E_+)$ comes from $H^2(M,U_+;\mathbb R)$. Therefore
$
    e(E_-)\smile e(E_+)
$
comes from
\begin{equation*}
H^4(M,U_-\cup U_+;\mathbb R)=H^4(M,M;\mathbb R)=0
\end{equation*}
by \eqref{Upm}. Hence
$
    e(E_-)\smile e(E_+)=0.
$
Using the Whitney product formula for the splitting \eqref{TMdecom}, we get
$
    e(TM)=e(E_-)\smile e(E_+)=0.
$
Therefore
$
    \chi(M)=0.
$


It remains to treat the case where $E_-$ and $E_+$ are not necessarily orientable.   Since $M$ is orientable and $TM=E_-\oplus E_+$, the Whitney formula gives
\begin{equation*}
   0=w_1(TM)=w_1(E_-)+w_1(E_+)
    \quad\text{in }H^1(M;\mathbb Z_2).
\end{equation*}
Since the coefficients are $\mathbb Z_2$, this is equivalent to
$
    w_1(E_-)=w_1(E_+).
$
Thus $E_-$ and $E_+$ have the same orientation double cover.

If both $E_-$ and $E_+$ are orientable, the preceding argument applies directly.  Otherwise, let
$
    \pi:\widetilde M\to M
$
be the two-fold covering associated with the nonzero class
\begin{equation*}
w_1(E_-)=w_1(E_+)\in H^1(M;\mathbb Z_2).
\end{equation*}
Then
$ \pi^*E_-$ and  $ \pi^*E_+$ are orientable rank two bundles over $\widetilde M$.  Moreover,
\begin{equation*}
    T\widetilde M=\pi^*TM=\pi^*E_-\oplus \pi^*E_+ .
\end{equation*}
Let
$\widetilde U_\pm:=\pi^{-1}(U_\pm).$
Since
$
    U_-\cup U_+=M,
$
we have
$
    \widetilde U_-\cup \widetilde U_+=\widetilde M.
$
On $\widetilde U_-$, the bundle $\pi^*E_-$ splits into two principal line subbundles, and hence
\begin{equation*}
   e(\pi^*E_-)|_{\widetilde U_-}=0
    \quad\text{in }H^2(\widetilde U_-;\mathbb R).
\end{equation*}
Similarly,
\begin{equation*}
    e(\pi^*E_+)|_{\widetilde U_+}=0
    \quad\text{in }H^2(\widetilde U_+;\mathbb R).
\end{equation*}
Thus the same relative cohomology argument as above gives
\begin{equation*}
   e(\pi^*E_-)\smile e(\pi^*E_+)=0
    \quad\text{in }H^4(\widetilde M;\mathbb R).
\end{equation*}
By the Whitney product formula,
\begin{equation*}
 e(T\widetilde M)
    =
    e(\pi^*E_-)\smile e(\pi^*E_+).
\end{equation*}
Therefore
$
    e(T\widetilde M)=0,
$
and hence
$
    \chi(\widetilde M)=0.
$
Since $\pi:\widetilde M\to M$ is a two-fold covering, the Euler characteristic is multiplicative under finite coverings:
$
    \chi(\widetilde M)=2\chi(M).
$
Consequently,
$
    \chi(M)=0.
$
This proves the lemma.
\end{proof}


\vspace{0.5cm}
\subsection{The equality case $|A|^2\equiv 12,\ f_4\equiv 68$}
\begin{prop}\label{prop:S12_f4_68}
Let \(M^4\subset \mathbb S^5\) be a closed connected oriented minimal hypersurface
satisfying
$
    |A|^2\equiv 12,
    \ 
    f_4\equiv 68.
$
Then $M$ is congruent to the Cartan minimal hypersurface.
\end{prop}

\begin{proof}
By Newton's identities,
$
    K=\lambda_1\lambda_2\lambda_3\lambda_4=1.
$
Then the principal curvatures are the roots of
\begin{equation}
x^4-6x^2-\frac{f_3}{3}x+1=0. \label{characterpoly}
\end{equation} 
Define
\begin{equation*}
   Y:=\dkh{p\in M:\ |f_3(p)|<8\sqrt{2\sqrt{3}} }.
\end{equation*}
By Lemma~\ref{lem:S12_f4_68_algebra}, on $Y$ the four principal curvatures are distinct.  Hence they may be ordered smoothly as
$
    \lambda_1<\lambda_2<\lambda_3<\lambda_4.
$
Let
\begin{equation*}
    D_i:=\prod_{j\ne i}(\lambda_i-\lambda_j),
    \qquad
    \Delta:=\prod_{1\le i<j\le4}(\lambda_i-\lambda_j).
\end{equation*}
We shall use a local ordered principal orthonormal frame
$
    e_1,\cdots,e_4,
   \ 
    h_{ij}=\lambda_i\delta_{ij}.
$
In this frame,
$
    h_{iik}=\lambda_{i,k}.
$
Since
\begin{equation*}
  \lambda_i^4-6\lambda_i^2-\frac{f_3}{3}\lambda_i+1=0,
\end{equation*}
differentiating with respect to \(e_k\) gives
\begin{equation*}
  D_i\lambda_{i,k}-\frac{\lambda_i}{3}f_{3,k}=0.
\end{equation*}
Therefore
\begin{equation}
   h_{iik}= \lambda_{i, k}=\frac{\lambda_i}{3D_i}f_{3, k}.
    \label{eq:h_iik_f_k}
\end{equation} 

For \(i<j\), define the \(3\)-forms
\begin{align*}
    \theta_{12}&:=\omega_3\wedge\omega_4\wedge\omega_{12},&
    \theta_{13}&:=\omega_4\wedge\omega_2\wedge\omega_{13},\\
    \theta_{14}&:=\omega_2\wedge\omega_3\wedge\omega_{14},&
    \theta_{23}&:=\omega_1\wedge\omega_4\wedge\omega_{23},\\
    \theta_{24}&:=\omega_3\wedge\omega_1\wedge\omega_{24},&
    \theta_{34}&:=\omega_1\wedge\omega_2\wedge\omega_{34}.
\end{align*}
Equivalently,
$\theta_{ij}=*(\omega_i\wedge\omega_j)\wedge\omega_{ij},
$ where $*$ is the Hodge star with respect to the oriented coframe $\omega_1,\cdots,\omega_4$.
Now define
\begin{equation*}
    \Psi
    :=
    -\sum_{i<j}
    \left(
        \lambda_i^2+\lambda_i\lambda_j+\lambda_j^2
    \right)\theta_{ij}.
\end{equation*}
The expression is independent of the local choice of ordered principal frame. Indeed, after the ordering of the simple principal curvatures is fixed, the only remaining ambiguity is the sign change
$
    e_i\mapsto s_i e_i,
    \ 
    (s_i=\pm1).
$
On an oriented overlap one has $s_1s_2s_3s_4=1$.  Under this change,
\begin{equation*}
   \omega_k\mapsto s_k\omega_k,
    \qquad
    \omega_{ij}\mapsto s_is_j\omega_{ij},
\end{equation*} 
and hence each $\theta_{ij}$ is unchanged.  Thus $\Psi$ is a globally defined $3$-form on $Y$.

We next compute $d\Psi$.

\begin{lemma}\label{lem:dPsi_detailed}
On $Y$, one has
\begin{equation}
  d\Psi
    =
    \sum_{r=1}^4
    G_r(\lambda) f_{3,r}^2\,dV,
    \label{eq:dPsi_positive}
\end{equation} 
where
\begin{equation}
G_r(\lambda)
    =
    \frac{2P(\lambda_r^2)}
         {9\lambda_r^2\Delta^2},
    \label{eq:G_r}
\end{equation}
and
\begin{equation}
    P(y)=3y^4-40y^3+162y^2-120y+27.
    \label{3.4} 
\end{equation} 
Moreover,
\begin{equation*}
    G_r(\lambda)>0
    \qquad
    \text{on }Y.
\end{equation*}
\end{lemma}

\begin{proof}
  At a point of $Y$, use an ordered principal orthonormal frame.  For $i\ne j$,  we have $h_{ij}\equiv0$, and hence
$dh_{ij}=0$.  Therefore, from \eqref{hijkdef}, we have
\begin{equation*}
\sum_k h_{ijk}\omega_k =(\lambda_i-\lambda_j)\omega_{ij}.
\end{equation*}
Hence
\begin{equation}
 \omega_{ij}
    =
    \sum_k
    \frac{h_{ijk}}{\lambda_i-\lambda_j}\omega_k.
    \label{eq:omega_ij_hijk}
\end{equation} 
Let first
$
    \Phi=\sum_{i<j}\phi_{ij}\theta_{ij}
$
be an arbitrary $3$-form of this type, where $\phi_{ij}=\phi_{ji}$ are smooth symmetric functions of the principal curvatures.  
Therefore,
\begin{equation*}
d\Phi=\sum_{i<j}d(\phi_{ij}\theta_{ij})=\sum_{i<j}
    \kh{d\phi_{ij}\wedge\theta_{ij}+\phi_{ij}d\theta_{ij}}.
\end{equation*}
 By using \eqref{eq:omega_ij_hijk}, we obtain,
\begin{align*}
d\phi_{12}\wedge \theta_{12}=&\phi_{12,k}\omega_k\wedge \omega_3\wedge \omega_4\wedge \frac{h_{12l}}{\lambda_1-\lambda_2} \omega_l\\
=&\frac{\phi_{12,1}h_{122}-\phi_{12,2}h_{121}}{\lambda_1-\lambda_2} \omega_1\wedge \omega_2\wedge \omega_3\wedge \omega_4\\
=&\frac{\phi_{12,1}h_{122}-\phi_{12,2}h_{121}}{\lambda_1-\lambda_2} dV.
\end{align*} 
Here $dV=\omega_1\wedge \omega_2\wedge \omega_3 \wedge \omega_4$ is the volume form.
We now calculate the other term
\begin{equation*}
\phi_{12}d\theta_{12} 
 =\phi_{12}\kh{d\omega_3\wedge \omega_4\wedge \omega_{12}-\omega_3\wedge d\omega_4\wedge \omega_{12}+\omega_3\wedge \omega_4 \wedge d\omega_{12}}
\end{equation*}
For simplicity, we write $c_{ij}^k=\dfrac{h_{ijk}}{\lambda_i-\lambda_j}$. By using \eqref{domegai}, \eqref{domegaij}, \eqref{Gausseq} and \eqref{eq:omega_ij_hijk}, a direct computation gives
\begin{align*}
d\omega_3\wedge \omega_4\wedge \omega_{12}=&\omega_{3j}\wedge\omega_j\wedge \omega_4\wedge\omega_{12}=c_{3j}^kc_{12}^l\omega_k\wedge \omega_j\wedge\omega_4\wedge \omega_l\\
=&\kh{c_{31}^2c_{12}^3-c_{31}^3c_{12}^2-c_{32}^1c_{12}^3
+c_{32}^3c_{12}^1 }dV,\\
-\omega_3\wedge d\omega_4\wedge \omega_{12}=&-\omega_3\wedge \omega_{4j}\wedge\omega_j\wedge \omega_{12}=-c_{4j}^kc_{12}^l \omega_3\wedge \omega_k\wedge \omega_j \wedge \omega_l\\
=&\kh{c_{41}^2c_{12}^4- c_{41}^4c_{12}^2-c_{42}^1c_{12}^4+c_{42}^4c_{12}^1 }dV,\\
\omega_3\wedge \omega_4 \wedge d\omega_{12}=&\omega_3\wedge \omega_4 \wedge \kh{\omega_{1j}\wedge\omega_{j2}-\frac12 R_{12ab}\omega_a\wedge \omega_b}\\
=&\omega_3\wedge \omega_4\wedge\kh{c_{1j}^kc_{j2}^l\omega_k\wedge \omega_l -R_{1212}\omega_1\wedge \omega_2}\\
=&\kh{c_{13}^1c_{32}^2-c_{13}^2c_{32}^1+c_{14}^1c_{42}^2-c_{14}^2c_{42}^1 -(1+\lambda_1\lambda_2)}dV.
\end{align*}
In summary, we have
\begin{align*}
d\theta_{12}=&\left[ \frac{2h_{123}^2}{(\lambda_3-\lambda_1)(\lambda_3-\lambda_2)} +\frac{2h_{124}^2}{(\lambda_4-\lambda_1)(\lambda_4-\lambda_2)}+ \frac{h_{122}h_{133}}{(\lambda_1-\lambda_2)(\lambda_1-\lambda_3)}\right.\\
&\ \ \ +\frac{h_{211}h_{233}}{(\lambda_2-\lambda_1)(\lambda_2-\lambda_3)}+\frac{h_{122}h_{144}}{(\lambda_1-\lambda_2)(\lambda_1-\lambda_4)}+\frac{h_{211}h_{244}}{(\lambda_2-\lambda_1)(\lambda_2-\lambda_4)}\\
&\ \ \ \left. -\frac{h_{311}h_{322}}{(\lambda_3-\lambda_1)(\lambda_3-\lambda_2)}-\frac{h_{411}h_{422}}{(\lambda_4-\lambda_1)(\lambda_4-\lambda_2)}-(1+\lambda_1\lambda_2)\right] dV
\end{align*} 
The other terms of $d\phi_{ij}\wedge\theta_{ij}$ and $\phi_{ij}d\theta_{ij}$ can be calculated directly in the same way. After these calculations, $d\Phi$ can be written as
\begin{equation}
   d\Phi
    =
    \left(
        \mathcal R_\Phi
        +
        \mathcal T_\Phi
        +
        \mathcal D_\Phi
    \right)dV.  \label{eq:general_dPhi}
\end{equation} 
Here $\mathcal R_\Phi$ is the curvature contribution,
\begin{equation}
   \mathcal R_\Phi
    =
    -\sum_{i<j}(1+\lambda_i\lambda_j)\phi_{ij}.
     \label{eq:curvature_contribution}
\end{equation} 
The term $\mathcal T_\Phi$ is the contribution of the components $h_{ijk}$ with three distinct indices.  For every three distinct indices $i,j,k$, the coefficient of $h_{ijk}^2$ is 
\begin{equation}
 \frac{2\left[
    (\lambda_i-\lambda_j)\phi_{ij}
    -
    (\lambda_i-\lambda_k)\phi_{ik}
    +
    (\lambda_j-\lambda_k)\phi_{jk}
    \right]}
    {(\lambda_i-\lambda_j)(\lambda_i-\lambda_k)(\lambda_j-\lambda_k)}.
     \label{eq:three_distinct_coeff}
\end{equation} 
Finally, $\mathcal D_\Phi$ contains only terms of the form
$
    h_{aar}h_{bbr}
$
with the same derivative direction \(r\).

We now apply this to our case
\begin{equation*}
    \phi_{ij} = -\left( \lambda_i^2+\lambda_i\lambda_j+\lambda_j^2\right)
    =-\frac{\lambda_i^3-\lambda_j^3}{\lambda_i-\lambda_j}.
\end{equation*}
Then for any three distinct $i,j,k$,
\begin{align*}
&(\lambda_i-\lambda_j)\phi_{ij}
-
(\lambda_i-\lambda_k)\phi_{ik}
+
(\lambda_j-\lambda_k)\phi_{jk}      \\
&=
-(\lambda_i^3-\lambda_j^3)
+
(\lambda_i^3-\lambda_k^3)
-
(\lambda_j^3-\lambda_k^3)           \\
&=0.
\end{align*}
Hence the term $\mathcal T_\Psi$ vanishes.

Next, by \eqref{eq:curvature_contribution}, the curvature term is
\[
\mathcal R_\Psi
=
\sum_{i<j}
(1+\lambda_i\lambda_j)
\left(
    \lambda_i^2+\lambda_i\lambda_j+\lambda_j^2
\right).
\]
Using
$
    \sum_i\lambda_i=0,\ 
    \sum_i\lambda_i^2=12,\ 
    \sum_i\lambda_i^4=68,
$
we compute
\begin{align*}
\sum_{i<j}(\lambda_i^2+\lambda_j^2)=&3\sum_i\lambda_i^2=36,\\
\sum_{i<j}\lambda_i\lambda_j=&-\frac12\sum_i\lambda_i^2
 =-6,\\
 \sum_{i<j}\lambda_i\lambda_j(\lambda_i^2+\lambda_j^2)
    =&\sum_i\lambda_i^3\sum_{j\ne i}\lambda_j
    =-\sum_i\lambda_i^4
    = -68,\\
\sum_{i<j}\lambda_i^2\lambda_j^2
    =&   \frac12\left[\left(\sum_i\lambda_i^2\right)^2
-\sum_i\lambda_i^4\right]=\frac12(144-68)=38.
\end{align*}
Therefore
\begin{equation*}
 \mathcal R_\Psi=36-6-68+38=0.
\end{equation*} 
Thus $d\Psi=\mathcal D_\Psi dV$.
By \eqref{eq:h_iik_f_k}, for each fixed derivative direction $r$, $\mathcal D_\Psi$ is a multiple of $f_{3,r}^2$.  Consequently,
we can assume
\begin{equation*}
d\Psi=\sum_{r=1}^4G_r(\lambda)f_r^2\,dV.
\end{equation*}    
It remains to identify $G_r$.  We compute $G_1$; the other coefficients are obtained by relabelling the indices.  Let
\begin{equation*}
    s_1=\lambda_2+\lambda_3+\lambda_4,\qquad
    s_2=\lambda_2\lambda_3+\lambda_2\lambda_4+\lambda_3\lambda_4,
\qquad
    s_3=\lambda_2\lambda_3\lambda_4.
\end{equation*}
Since
$
    \sum_i\lambda_i=0,\ 
    \sum_i\lambda_i^2=12,\ 
    \prod_i\lambda_i=1,
$
we have
\begin{equation}\label{s123}
    s_1=-\lambda_1,\qquad
    s_2=\lambda_1^2-6,\qquad
    s_3=\frac{1}{\lambda_1}.
\end{equation}
A direct simplification of the coefficient of $f_{3,1}^2$ gives
$
    G_1=\frac{N_1}{9\Delta^2},
$
where
\begin{align*}
N_1=&-4\lambda_1^2s_1^3s_3+\lambda_1^2s_1^2s_2^2
+6\lambda_1^2s_1s_2s_3-\lambda_1^2s_2^3                                    \\
    & +18\lambda_1 s_1s_3^2-6\lambda_1 s_2^2s_3
    +3s_1^2s_3^2-9s_2s_3^2 .
\end{align*} 
Substituting \eqref{s123} yields
\begin{equation*}
    N_1 =\frac{2}{\lambda_1^2}\left(3\lambda_1^8-40\lambda_1^6+162\lambda_1^4
-120\lambda_1^2+27\right).
\end{equation*}

Hence
$G_1=\frac{2P(\lambda_1^2)}{9\lambda_1^2\Delta^2}.
$
The same computation with the role of \(\lambda_1\) replaced by \(\lambda_r\)
gives
\begin{equation*}
G_r
    =
    \frac{2P(\lambda_r^2)}
         {9\lambda_r^2\Delta^2},
    \qquad r=1,\ldots,4.
\end{equation*}

Finally,
\begin{equation*}
 P(y)=3y^4-40y^3+162y^2-120y+27 =3\left(y^2-\frac{20}{3}y+3\right)^2 + \frac{32}{3}y^2>0.
\end{equation*}
On $Y$,  the four principal curvatures are distinct and non-zero, so
$
    \Delta^2>0.
$
Therefore $G_r(\lambda)>0$. The proof is complete.
\end{proof}

By using the definition of $\Psi$ and \eqref{eq:omega_ij_hijk}, a direct calculation gives
\begin{equation}
    df_3\wedge\Psi=\sum_{r=1}^4 f_{3,r} \omega_r\wedge \Psi=\sum_{r=1}^4U_r(\lambda)f_{3,r}^2\,dV,
    \label{eq:df_Psi_U}
\end{equation}
where
\begin{equation}
 U_r(\lambda)
    =
    \sum_{j\ne r}
    \frac{
    -\left(
        \lambda_r^2+\lambda_r\lambda_j+\lambda_j^2
    \right)\lambda_j
    }
    {3(\lambda_r-\lambda_j)D_j}.
    \label{eq:U_r}
\end{equation} 
\begin{lemma}\label{lem:endpoint_U}
There exists a constant $C>0$ such that
$
    U_r\ge -C
$
near the endpoint $f_3=-8\sqrt{2\sqrt{3}}$, for every $r$, and
$
    U_r\le C
$
near the endpoint $f_3=8\sqrt{2\sqrt{3}}$, for every $r$.
\end{lemma}

\begin{proof}
We prove the statement near $f_3=-8\sqrt{2\sqrt{3}}$.  By
Lemma~\ref{lem:S12_f4_68_algebra}, as $f_3\to -8\sqrt{2\sqrt{3}}$, the limiting principal
curvature configuration is
\begin{equation*}
    \lambda_1<\lambda_2<0<\lambda_3=\lambda_4.
\end{equation*}
Write
$
    \delta=\lambda_4-\lambda_3\to0^+.
$
For $U_3$, the only singular term in \eqref{eq:U_r} is the one with
$j=4$:
\begin{equation*}
    \frac{
    -(\lambda_3^2+\lambda_3\lambda_4+\lambda_4^2)\lambda_4
    }
    {3(\lambda_3-\lambda_4)D_4}.
\end{equation*}
Since
$
    D_4
    =
    (\lambda_4-\lambda_1)(\lambda_4-\lambda_2)(\lambda_4-\lambda_3),
$
we have
\begin{equation*}
    (\lambda_3-\lambda_4)D_4
    =
    -(\lambda_4-\lambda_1)(\lambda_4-\lambda_2)(\lambda_4-\lambda_3)^2
    <0.
\end{equation*}
The numerator tends to a negative number because
$
    \lambda_3,\lambda_4>0.
$
Therefore this singular term tends to $+\infty$.  The remaining terms in $U_3$ have finite limits.  Hence
$ U_3\to+\infty$ as $f_3\to -8\sqrt{2\sqrt{3}}.$
The same argument gives
$
    U_4\to+\infty.
$

For $U_1$, the potentially singular terms are those with
$j=3$ and $j=4$.  
A direct computation yields
\begin{equation*}
U_1=\frac{-\left(\lambda_1^2+\lambda_1\lambda_2+\lambda_2^2 \right)\lambda_2}{3(\lambda_1-\lambda_2)D_1}+ \frac{F_1(\lambda_3)-F_1(\lambda_4)}{\lambda_3-\lambda_4},
\end{equation*}
where
\begin{equation*}
F_1(t)=\frac{t(\lambda_1^2+\lambda_1t+t^2)}{3(t-\lambda_1)^2(t-\lambda_2)}.
\end{equation*}
Note that $F_1$ is smooth near the point that $f_3=-8\sqrt{2\sqrt{3}}$. Therefore,
\begin{equation*}
\lim_{\delta\to 0^+} \frac{F_1(\lambda_3)-F_1(\lambda_4)}{\lambda_3-\lambda_4}=F_1^\prime(\alpha), 
\end{equation*}
where $\alpha=\sqrt{1+\frac{2\sqrt3}{3}}$ is defined in the proof of Lemma \ref{lem:S12_f4_68_algebra}. Thus, $U_1$ has finite limit as $f_3\to -8\sqrt{2\sqrt{3}}$,  and  bounded  near $f_3=-8\sqrt{2\sqrt{3}}$.
A similar argument implies $U_2$ are also bounded
near $f_3=-8\sqrt{2\sqrt{3}}$.  Consequently, $U_r$ are all bounded from below near $f_3=-8\sqrt{2\sqrt{3}}$.

The proof near $f_3=8\sqrt{2\sqrt{3}}$ is analogous.  In that case the limiting configuration is
\begin{equation*}
  \lambda_1=\lambda_2<0<\lambda_3<\lambda_4.
\end{equation*}
After similar arguments as above, $U_1$ and $U_2$ tend to $-\infty$, while $U_3$ and $U_4$ are bounded near $f_3=8\sqrt{2\sqrt{3}}$. Therefore, all the $U_r$ are bounded from above near $f_3=8\sqrt{2\sqrt{3}}$.
\end{proof}

We now finish the proof of Proposition~\ref{prop:S12_f4_68}.  For simplicity, we set $T=8\sqrt{2\sqrt{3}}$. For $\varepsilon>0$  small enough, choose a smooth
cutoff function
$
    \eta_\varepsilon:\mathbb R\to[0,1]
$
such that
\begin{align*}
&\eta_\varepsilon=0,  \quad
    \text{on}    \ \ [-T,-T+\varepsilon/3]\cup[T-\varepsilon/3,T],\\
&\eta_\varepsilon=1,
    \quad
    \text{on }
    [-T+\varepsilon,T-\varepsilon],\\
&\eta_\varepsilon'\ge0,
    \quad\text{on } [-T+\varepsilon/3, -T+\varepsilon],
    \\
   & \eta_\varepsilon'\le0,
    \quad\text{on }[T-\varepsilon, T-\varepsilon/3].
\end{align*} 
By the standard construction of the cutoff function, there exists a constant $C_\eta>0$, independent of $\varepsilon$, such that
\begin{equation}
   |\eta_\varepsilon'(t)|\le \frac{C_\eta}{\varepsilon}. \label{Ceta}
\end{equation}
 
Since $(\eta_\varepsilon\circ f_3)\Psi$ has compact support in $Y$, Stokes' theorem gives
\begin{equation*}
  0= \int_Y d\big((\eta_\varepsilon\circ f_3)\Psi\big).
\end{equation*}
Hence
\begin{equation}\label{etaepsilonf3}
 \int_Y(\eta_\varepsilon\circ f_3)d\Psi
=-\int_Y(\eta_\varepsilon'\circ f_3)\,df_3\wedge\Psi.
\end{equation}
On the other hand, by \eqref{eq:df_Psi_U}, we have
\begin{equation*}
-\int_Y(\eta_\varepsilon'\circ f_3)\,df\wedge\Psi
=
-\int_Y(\eta_\varepsilon'\circ f_3)
\sum_{r=1}^4U_r f_{3,r}^2\,dV.
\end{equation*}
On  $[-T+\varepsilon/3, -T+\varepsilon]$, $\eta_\varepsilon'\ge0$ and
$U_r\ge -C$ by Lemma \ref{lem:endpoint_U}. Hence
\begin{equation*}
  -\eta_\varepsilon' U_r f_{3,r}^2
    \le
    C|\eta_\varepsilon'|f_{3,r}^2.
\end{equation*}
On $[T-\varepsilon, T-\varepsilon/3]$, $\eta_\varepsilon'\le0$ and
$U_r\le C$. Hence again
\begin{equation*}
    -\eta_\varepsilon' U_r f_{3,r}^2
    \le
    C|\eta_\varepsilon'|f_{3,r}^2.
\end{equation*}
Therefore
\begin{equation*}
-\int_Y(\eta_\varepsilon'\circ f_3)\,df\wedge\Psi
\le
C\int_{\{|f_3|>T-\varepsilon\}}
|\eta_\varepsilon'\circ f_3|\,|\nabla f_3|^2\,dV.
\end{equation*}
By Lemma \ref{lem:dPsi_detailed}, $d\Psi\ge0$. Combining this fact with $\eta_\varepsilon\ge0$ and \eqref{etaepsilonf3}, we have 
\begin{equation}
 0\le\int_Y(\eta_\varepsilon\circ f_3)d\Psi \le C \int_{\{|f_3|>T-\varepsilon\}} |\eta_\varepsilon'\circ f_3|\,|\nabla f_3|^2\,dV.
   \label{eq:cutoff_1}
\end{equation} 

Since
\begin{equation*}
    \operatorname{supp}(\eta_\varepsilon'\circ f_3)
    \subset
    \{T-\varepsilon\le |f_3|\le T-\varepsilon/3\},
\end{equation*}
and the upper bound \eqref{Ceta}. Lemma~\ref{lem:layer_estimate} gives
\begin{equation}
    \lim_{\varepsilon\to0}
    \int_M
    |\eta_\varepsilon'\circ f_3|\,|\nabla f_3|^2\,dV
    =
    0.   \label{eq:layer_limit}
\end{equation}

 Combining \eqref{eq:cutoff_1} and \eqref{eq:layer_limit}, we obtain
\begin{equation}
    \lim_{\varepsilon\to0}
    \int_Y(\eta_\varepsilon\circ f_3)d\Psi
    =
    0.
 \label{eq:dPsi_limit}
\end{equation}

Now fix $0<\varepsilon'<T$.  On the compact set
\begin{equation*}
    Y_{\varepsilon'}
    :=
    \{p\in M:\ |f_3(p)|\le T-\varepsilon'\},
\end{equation*}
the coefficients $G_r(\lambda)$ in Lemma~\ref{lem:dPsi_detailed} have a
positive lower bound.  Therefore there exists $c_{\varepsilon'}>0$ such that
\begin{equation*}
    d\Psi
    \ge
    c_{\varepsilon'}|\nabla f_3|^2\,dV
    \qquad
    \text{on }Y_{\varepsilon'}.
\end{equation*}
For all sufficiently small $\varepsilon<\varepsilon'$, one has
$
    \eta_\varepsilon\equiv1
 $
   on $Y_{\varepsilon'}$.
Hence
\begin{equation*}
    0
    \le
    c_{\varepsilon'}
    \int_{Y_{\varepsilon'}}|\nabla f_3|^2\,dV
    \le
    \int_Y(\eta_\varepsilon\circ f_3)d\Psi.
\end{equation*}
Letting $\varepsilon\to0$ and using \eqref{eq:dPsi_limit}, we get
\begin{equation*}
    \int_{Y_{\varepsilon'}}|\nabla f_3|^2\,dV=0.
\end{equation*}
Since $\varepsilon'>0$ is arbitrary,
\begin{equation*}
    \nabla f_3=0
    \qquad
    \text{on }Y.
\end{equation*}

If $Y=\emptyset$, then $f_3$ takes values only in $\{-T,T\}$.  Since $M$ is connected and $f_3$ is continuous, $f_3$ is constant.

If $Y\ne\emptyset$, then $f_3$ is constant on each connected component of $Y$.  Since this constant lies in $(-T,T)$, such a component cannot have boundary in $M$.  Hence it is both open and closed in $M$.  By the
connectedness of $M$, we have
$
    Y=M,
$
and $f=f_3$ is constant on $M$.

Thus
\begin{equation*}
    f_1\equiv0,\qquad
    f_2\equiv12,\qquad
    f_3\equiv\text{constant},\qquad
    f_4\equiv68.
\end{equation*}
Therefore, the principal curvatures are constants since they are all the roots of the characterized polynomial \eqref{characterpoly}.  Hence $M$ is isoparametric.  Since
$
    |A|^2\equiv12,
$
the classification of minimal isoparametric hypersurfaces in $\mathbb S^5$
implies that $M$ is congruent to the Cartan minimal hypersurface.
This proves Proposition~\ref{prop:S12_f4_68}.
\end{proof}


\vspace{0.5cm}
\subsection{Completion of the proof}
\begin{proof}[Proof of Theorem~\ref{thm:main}]
First, by Lemma \ref{lem:umbilic_implies_totally_geodesic}, if $M$ has an umbilic point, then $M$ is totally geodesic. Assume $M$ has no umbilic point, then by Lemma \ref{lem:no_umbilic_chi_zero}, we have $\chi(M)=0$. Consequently, by Lemma \ref{lem:GBC}, we obtain
\begin{align*}
|A|^2\equiv 12, \quad f_4=\frac{17}{36} |A|^4\equiv 68.
\end{align*}
The conclusion follows immediately by Proposition \ref{prop:S12_f4_68}.
\end{proof}

\vspace{2cm}

\providecommand{\bysame}{\leavevmode\hbox to3em{\hrulefill}\thinspace}
\providecommand{\MR}{\relax\ifhmode\unskip\space\fi MR }


\end{document}